\documentclass{article}
\usepackage[english]{babel}
\usepackage[T1]{fontenc}
\usepackage[utf8]{inputenc}
\usepackage{graphicx}
\usepackage{hyperref}
\usepackage[table,xcdraw]{xcolor}
\usepackage{lmodern}
\usepackage{amsmath, amssymb, mathtools, amsthm, thm-restate}
\usepackage{booktabs}
\usepackage{xspace}
\newtheorem{problem}{Problem}
\newtheorem{theorem}{Theorem}
\newtheorem{lemma}[theorem]{Lemma}
\newtheorem{observation}[theorem]{Observation}
\newtheorem{corollary}[theorem]{Corollary}

\newcommand{\ra}{\xrightarrow[]{}}
\newcommand{\nauty}{{\tt nauty}\xspace}
\newcommand{\geng}{{\tt geng}\xspace}
\newcommand{\hogid}[1]{\href{https://houseofgraphs.org/graphs/#1}{\texttt{HoG<#1>}}}
\newcommand{\undercite}[2]{\underset{\text{\tiny\cite{#1}}}{#2}}
\newcommand{\underlem}[2]{\underset{\text{\tiny Lem.\ \ref{#1}}}{#2}}

\newcommand{\llsym}{\bowtie}
\newcommand{\circsym}{\circledast}
\newcommand{\extsym}{\boxplus}
\newcommand{\modsym}{\dagger}
\newcommand{\filtersym}{\star}

\newcommand{\ptt}{\overline{P_2 \cup P_3}}

\DeclareMathOperator{\Aut}{Aut}

\definecolor{myblue}{HTML}{DAE8FC}
\definecolor{mygray}{HTML}{A8A8A8}
\definecolor{myorange}{HTML}{ffcb2f}
\definecolor{mygreen}{HTML}{00d300}
\definecolor{mywhite}{HTML}{ffffff}

\usepackage{tikz}
\tikzset{unlabeled_vertex/.style={inner sep=2.5pt, outer sep=0pt, circle, fill}}
\tikzset{edge_color0/.style={color=black,line width=1.2pt,opacity=0.5}}
\tikzset{edge_color1/.style={color=red,  line width=1.2pt,opacity=1}} 
\tikzset{edge_color2/.style={color=blue, line width=1.2pt,opacity=1}}
\tikzset{edge_color3/.style={color=green!80!black,line width=1.2pt,opacity=1}}

\title{On Small Folkman Graphs Arrowing $K_2$ or $K_3$}
\author{
Zohair Raza Hassan\\[-0.2ex]
\small Department of Computer Science, \\[-0.6ex] {\small \tt zh5337@rit.edu}\\[-0.6ex]
\small Rochester Institute of Technology, Rochester, NY 14623, USA\\[1.0ex]\and
Stanis{\l}aw Radziszowski\\[-0.3ex]
\small Department of Computer Science, \\[-0.6ex] {\small \tt spr@cs.rit.edu}\\[-0.6ex]
\small Rochester Institute of Technology, Rochester, NY 14623, USA\\[1.0ex]\and
Steven Van Overberghe\\[-0.2ex]
\small Department of Mathematics, Computer Science and Statistics, \\[-0.6ex] {\small \tt Steven.VanOverberghe@ugent.be}\\[-0.6ex]
\small Krijgslaan 299, 9000 Gent, Belgium\\[1.0ex]
}

\date{}

\begin{document}

\maketitle

\begin{abstract}
For a graph $G$ and integers $a_i \geq 1$, we say that $G \ra (a_1, \ldots, a_k)^v$ if in any $k$-coloring of $G$'s vertices there exists a monochromatic $a_i$-clique for some color $i \in \{1,\ldots,k\}$.
$G \ra (a_1, \ldots, a_k)^e$ is defined similarly, but for edge colorings. 
The Folkman number $F_v(a_1, \ldots, a_k; H)$ is the smallest number of vertices for which an $H$-free graph arrowing $(a_1, \ldots, a_k)^v$ exists. 
$F_e(a_1, \ldots, a_k; H)$ is defined similarly for edge-arrowing. 

In this work, we present new bounds for Folkman numbers where $a_i \in \{2,3\}$ and $k \leq 4$, while avoiding $K_n$, $J_n$, for $n \in \{4,5,6\}$, where $K_n$ is the complete graph on $n$ vertices and $J_n$ is $K_n$ missing an edge. We also present results for $C_4$-free and $W_5$-free graphs, where $C_4$ is the cycle on four vertices and $W_5$ is the wheel graph on five vertices. 
Notably, we prove the existence of $F_e(3,3;W_5)$, leaving only one graph, $\ptt$, on five vertices for which the existence problem of $F_e(3,3;H)$ remains open. We provide some theoretical results that should aid in uncovering the existence of $F_v(3,3;\ptt)$.

Our new bounds are the result of a variety of methods involving filters, extension, semi-polycirculant graphs, locally linear graphs, and the modification of special graphs. 
Most of our bounds are from the semi-polycirculant graph generator, showcasing its efficacy for finding witness Folkman graphs. 
\end{abstract}

\section{Introduction}

\subsection{Notation and Terminology}

\noindent \textbf{Graphs.} All graphs discussed in this work are simple and undirected. The vertex and edge sets of a graph $G$ are denoted by $V(G)$ and $E(G)$, respectively. As is standard, we use $\alpha(G)$ to denote the order of the largest independent set in $G$, and $\chi(G)$ to denote its chromatic number. $K_n$, $C_n$, and $W_n$ refer to the complete, cycle, and wheel graphs on $n$ vertices. We use $J_n$ to denote the complete graph on $n$ vertices missing an edge, $K_n - e$.

\noindent \textbf{Arrowing.} 
Let $a_1, a_2, \ldots, a_k$ be positive integers, and let $G$ be a graph.  
We say that $G \ra (a_1,  \ldots, a_k)^v$ if any vertex-coloring of $G$ with $k$ colors has a monochromatic copy of $K_{a_i}$ for some color $i \in \{1,2,\ldots,k\}$. Note that the `$v$' in the superscript denotes the fact that we are coloring vertices. Edge arrowing is defined similarly: 
$G \ra (a_1,  \ldots, a_k)^e$ if any edge-coloring of $G$ with $k$ colors has a monochromatic copy of $K_{a_i}$ for some color $i \in \{1,2,\ldots,k\}$.

\noindent \textbf{Ramsey and Folkman Numbers.} For positive integers $a_1,  \ldots a_k$, the Ramsey number $R(a_1,  \ldots, a_k)$ is the order of the smallest complete graph, $n$, such that $K_n \ra (a_1,  \ldots, a_k)^e$. 
Folkman numbers are generalizations of Ramsey numbers, wherein we also avoid a specified subgraph. 

The vertex Folkman number
$F_v(a_1,  \ldots, a_k; H)$ is the smallest number of vertices of any graph $G$ such that $G$ is $H$-free and $G \ra (a_1,  \ldots, a_k)^v$. 
Edge Folkman numbers, $F_e(a_1,  \ldots, a_k; H)$, are defined similarly.
The set of Folkman graphs
$\mathcal{F}_v(a_1,  \ldots, a_k; H; n)$ is the set of all graphs $G$ on $n$ vertices that are $H$-free and arrow $(a_1, \ldots, a_k)^v$. The set $\mathcal{F}_e(a_1,  \ldots, a_k; H; n)$ is defined similarly for edge arrowing.

As is standard, we may use exponents to state that the same graph is being arrowed in multiple colors: e.g., $F_v(2^3 ; 4)$ is equivalent to $F_v(2,2,2;4)$.

Finally, we define some terminology which will be useful in describing our witness graphs. An $H$-free graph $G$ is called maximal $H$-free if the addition of any edge creates a subgraph $H$ in $G$. 
A graph $G$ for which $G \ra (a_1, \ldots, a_k)$ is called minimal if after the
deletion of any edge this arrowing does not hold. 
If $G$ is maximal and minimal,
it is referred to as bicritical.

\subsection{Background and Prior Work}

Ramsey numbers have been studied extensively since their introduction in 1930, with a special focus on Ramsey numbers for two colors. 
We refer the reader to a survey by Radziszowski~\cite{ds1} for an extensive overview of the subject. 
Folkman numbers have garnered lesser but still significant interest.
Unlike Ramsey numbers, Folkman numbers are not guaranteed to exist for all combinations of graphs in the argument; for example, there is no $J_4$-free graph $G$ such that $G \ra (3,3)^e$ and thus $F_e(3,3;J_4)$ does not exist---this is easy to see after observing that each edge belongs to at most one $K_3$ in any $J_4$-free graph. 

The most recent work on finding new vertex Folkman numbers was by Hassan et al., who explored Folkman numbers for $J_4$-free graphs~\cite{HassanJNRX23}. They showed the existence of $F_v(3^k;J_4)$ for $k \geq 3$ and provided exact values and bounds for several numbers of the form $F_v(2^i, 3^j; J_4)$.
Complete graphs are the most explored class within the study of Folkman numbers~\cite{Folk70,Nen84,Nen00,Nen03,LRU01,Nen09,Fv445,Nen07,FizGM2013a}. 
Some notable examples are: (1) triangle-free Folkman numbers, particularly  $F_v(2^k;3)$ which correspond to the smallest order of a triangle-free graph with chromatic number greater than $k$~\cite{DBLP:journals/jgt/Goedgebeur20}, and (2) $F_e(3,3;4)$ the value of which has eluded researchers for decades; the best known bounds are $21 \leq F_e(3,3;4) \leq 786$~\cite{BikovN20,lange2014MaxCut}.

Folkman numbers of the form $F_v(a_1, \ldots, a_k;K_s)$ have been studied extensively. 
It is easy to see that finding these numbers becomes more difficult as $s$ decreases, but special values of $a_1, \ldots, a_k$ and $s$ give rise to interesting problems as shown below based on the values $m := 1 + \sum_{i=1}^k (a_i-1)$ and $p := \max(a_1, \ldots, a_k)$. We mention two results relevant to this work as they will be useful later:

\begin{lemma}[\cite{luczak1996note,luczak2001minimal}]
\label{l:k_m-equality}
        If $m \geq p+1$, then
    $F_v(a_1, \ldots, a_k ; K_m) = m+p$.
\end{lemma}

\begin{lemma}[\cite{nenov2002class}]
\label{l:k_m-1-equality}
        If 
    $p = 3$ and $m \geq 6$, 
    then $F_v(a_1, \ldots, a_k ; K_{m-1}) = m+6$.
\end{lemma}

Folkman numbers of the form 
$F_v(a_1, \ldots, a_k ; K_{m-1})$ 
were studied extensively by Bikov, who took a systematic approach based on the value of $\max(a_1, \ldots, a_k)$ \cite{bikov2018}.
We note that Bikov's work is also a great resource for an extensive overview on Folkman numbers.
We also refer the interested reader to a book by Soifer~\cite{soifer2024new} for more background on the history of Folkman numbers as well as many other interesting problems related to graph coloring.

\subsection{Summary of New Results}

In this work, we present new bounds on Folkman numbers
arrowing $K_2$ and $K_3$, while avoiding graphs up to six vertices. Namely, we explore avoiding $K_k$'s, $J_k$'s for $k \in \{4,5,6\}$, as well as $C_4$ and $W_5$.

\begin{table}[t]
\centering
\begin{tabular}{@{}rccccccc@{}}
\toprule
$H$                           & $J_4$     & $K_4$               & $J_5$     & $K_5$     & $J_6$     & $K_6$            \\ \midrule
$F_v(2,3;H)$                    & $\undercite{HassanJNRX23}{14}$        &  $\underlem{l:k_m-equality}{7}$                   &  $4$         & {$4$}         &  {$4$}         &  {$4$}                \\
$F_v(2,2,3;H)$ &       $\undercite{HassanJNRX23}{20 \text{--} {\textbf{36}^\circsym}}$ & $\undercite{coles2006computing}{14}$                  &  $\textbf{10}^\filtersym$        &  {$\underlem{l:k_m-equality}{8}$}         &  {$5$}         &  {$5$}               \\
$F_v(3,3;H)$   & $\textbf{23}^{\llsym} \text{--} \textbf{45}^{\circsym}$ & $\undercite{khadzhiivanov1983example,PiwakowskiRU99}{14}$                  &  $\textbf{11}^\filtersym$       &  $\underlem{l:k_m-equality}{8}$         &  $5$         &  $5$                \\
$F_v(2,3,3;H)$ &                      & $\undercite{bikov2017edge}{20 \text{--} 24}$ & {$ \textbf{15}^\extsym \text{--}\textbf{18}^\circsym$ } &  $\underlem{l:k_m-1-equality}{12}$        &  {$\textbf{12}^\filtersym$}        &  {$\underlem{l:k_m-equality}{9}$}                \\
$F_v(3,3,3;H)$ &                                &  {$\leq \textbf{51}^\circsym$}           &  {$\leq \textbf{32}^{\circsym \modsym}$} &   {$\leq \textbf{21}^\circsym$} & $\textbf{15}^{\extsym \circsym}$ &  $\underlem{l:k_m-1-equality}{13}$ \\ \midrule
$F_e(3,3;H)$             & $\infty$  &   $\undercite{BikovN20,lange2014MaxCut}{21 \text{--} 786}$                  &  {$\leq \textbf{43}^\modsym$} & $\undercite{khadzhiivanov1983example,PiwakowskiRU99}{15}$        & $\textbf{11}^\filtersym$        & $\undercite{graham1968edgewise}{8}$                \\ \bottomrule
\end{tabular}
\caption{
This table summarizes new and previously known results for Folkman numbers arrowing $K_2$ and $K_3$ with up to three colors avoiding $J_k$ and $K_k$ for $k \in\{4,5,6\}$. 
New results have been boldfaced.
References for previously known results are provided below the value. Markings next to the bounds indicate how the bound was found according to the legend below. 
For small numbers with no reference or marking, the bounds are easy to see: (1) For $H \in \{J_5, K_5, J_6, K_6\}$, $F_v(2,3;H) = 4$ follows from the fact that $K_4$ is $H$-free and $K_4 \ra (2,3)^v$. (2) For $H \in \{J_6, K_6\}$, $F_v(2,2,3;H) = F_v(3,3;H) = 5$ follows from the fact that $K_5$ is $H$-free, and $K_5 \ra (2,2,3)^v$ and $K_5 \ra (3,3)^v$.
\\
{\footnotesize
$\filtersym$ Result obtained by generating graphs with \geng filters. \\
$\extsym$ Result obtained by extending smaller graphs. \\
$\circsym$ Result obtained by generating semi-polyciruclant graphs. \\
$\llsym$ Result obtained by generating locally linear graphs. \\
$\modsym$ Result obtained by using/modifying a special graph.\\
}
}
\label{tab:main}
\end{table}

\begin{table}[t]
\centering
\begin{tabular}{@{}rc@{}}
\toprule
Folkman Number     & Bounds    \\ \midrule

$F_v(2,3;C_4)$ &  $\textbf{17}^\filtersym$ \\
$F_v(3,3;C_4)$  &  $ \textbf{30}^{\llsym} \text{--} \textbf{63}^{\modsym} $ \\ \midrule
$F_e(3,3;W_5)$     & $\leq \textbf{64}^{\modsym}$ \\
\midrule
$F_v(2,3,3,3;K_5)$ & {$\leq \textbf{32}^\circsym$} \\
$F_v(3,3,3,3;K_6)$ &  $\leq \textbf{30}^\circsym$ \\ 
\bottomrule
\end{tabular}
\caption{This table summarizes new results for Folkman numbers arrowing $K_2$ and $K_3$ with up to four colors, avoiding $C_4$, $W_5$, $K_5$, and $K_6$. Results are marked according to the legend in Table~\ref{tab:main}.}
\label{tab:my-table2}
\end{table}

Our upper bounds are obtained by finding witness graphs through various methods, such as 
studying existing well-known graphs, and generating polycirculant graphs. Our lower bounds are obtained from 
generating locally linear graphs.
We provide links for witness graphs on the online repository House of Graphs. Our results are summarized in Tables~\ref{tab:main} and~\ref{tab:my-table2}.
We also
answer an open question posed by Hassan et al.~\cite{HassanJNRX23}, who asked whether a $(3,3)^e$ arrowing graph without $W_5$ exists; we show that $F_e(3,3;W_5) \leq 64$. Now, the only graph $H$ on five vertices for which the existence of $F_e(3,3;H)$ is unknown is $H = \ptt$. We prove the following facts about any minimal witness to this Folkman number:

\begin{restatable}{thm}{pttThm}
\label{thm:ptt}
Suppose $F_e(3,3;\ptt)$ exists. Then, a minimal arrowing graph $G$ has the following properties:
    \begin{enumerate}
        \item $G$ must be $K_4$-free.
        \item Each edge in $G$ must belong to at least two triangles. Equivalently, for any vertex $u \in V(G)$, all $v$ in $G[N(u)]$ have degree $\geq 2$ within $G[N(u)]$.
        \item For any vertex $u \in V(G)$, $G[N(u)]$ must be $K_3$-free and $C_4$-free. 
        \item Suppose $u$ is a vertex in $G$. Then, for each edge $e$ in $G[N(u)]$, there exists a vertex $v \in X = V(G) - N(u) - {u}$ that both of $e$’s endpoints are adjacent to.
        Moreover, for any two edges in $G[N(u)]$, the corresponding vertices $v \in X$ are distinct.
        \item Every vertex $u \in V(G)$ has degree $ \geq 8$.
    \end{enumerate}
\end{restatable}

We note that the first point in the theorem above gives us the following corollary, which speaks to the difficulty of finding this number.

\begin{corollary}
    $F_e(3,3;4) \leq F_e(3,3;\ptt)$
\end{corollary}

The rest of our paper is organized as follows.
We explain how polycirculant and locally linear graphs are generated in Sections~\ref{sec:polycirc} and~\ref{sec:locallin}, respectively. We discuss the results presented in Tables~\ref{tab:main} and~\ref{tab:my-table2} in Section~\ref{sec:results}. Our theoretical results on $F_e(3,3;\ptt)$ listed in Theorem~\ref{thm:ptt} are presented in Section~\ref{sec:ptt}.
Finally, we conclude in Section~\ref{sec:conclude}. The code of the algorithms used in our paper is available at \url{https://github.com/Steven-VO/vertexFolkman}.

\section{Generating Polycirculant Graphs}
\label{sec:polycirc}

\subsection{Definitions}

Polycirculant graphs are a class of highly symmetric graphs which have recently been explored to find witness graphs for Ramsey numbers~\cite{DBLP:journals/dam/GoedgebeurO22}.
We formally define these below.
An automorphism $f : V(G) \xrightarrow{} V(G)$ of a graph $G$ is a bijection that preserves the adjacency of vertices. We define $\Aut(G)$ as the set of automorphisms of $G$. Every automorphism $\theta\in \Aut(G)$ induces a partition of the vertex set in the following way: \( v\sim w \Leftrightarrow \exists n\in\mathbb{N} \text{ s.t. } \theta^n(v)=w \). The partition classes are called the \textit{orbits} of $V$ under $\theta$.

A graph is called \textit{polycirculant} (also known as block-circulant) if it has a non-trivial automorphism $\theta$ such that all vertex orbits under $\theta$ have the same size. If there are $k$ such vertex orbits, we call it $k$-polycirculant.
These graphs have been used succesfully in the past to obtain new lower bounds for Ramsey numbers~\cite{Exoo_1998,DBLP:journals/dam/GoedgebeurO22,DBLP:journals/dm/Wesley26}.

1-polycirculant graphs are also know as \textit{circulant} graphs.
These can also be defined using an integer $n$ and a set of integers $D$ like so: let $V(G) = \{0,1,2,\ldots, n-1\}$ and for every $i \in V(G)$, the $i^{th}$ vertex of $G$ is adjacent to the vertices labeled $(i-j) \mod n$ and $(i+j) \mod n$ for each $j \in D$. Less formally, $G$ has an automorphism that is a cyclic permutation of its vertices.
For example, each cycle graph $C_n$ is circulant.

In this work, we also used a generalization of polycirculant graphs where the orders of orbits do not
have to be the same. A graph is called \textit{semi-polycirculant} if there is an automorphism \(\theta\) such that for every two sizes of vertex orbits $a$ and $b$, we have $a | b$ or $b|a$.  For example, if the vertex orbits have sizes \([20,5,5]\), this can be thought of as a circulant graph on 20 vertices ``rotating around'' two circulant graphs on 5 vertices. 
The array $[20,5,5]$ is also called a block-structure, i.e., the vertices of the graph can be arranged into blocks of sizes 20, 5, and 5, where each block induces a circulant graph.

\subsection{Generation}

Given a block-structure $B$, and graphs $H_1$ and $H_2$, the algorithm generates all semi-polycirculant graphs $G$ adhering to said block-structure $B$, where $G$ is $H_1$-free and the complement of $G$ is $H_2$-free.
This is done using a simple backtracking algorithm that decides the presence of each edge-class under $B$. While this sounds computationally expensive, many steps can be taken to avoid generating isomorphic copies of the same graph due to the highly symmetric structure of semi-polycirculant graphs, allowing for a more efficient algorithm. More details on how these isomorphs are avoided can be found in \cite{DBLP:journals/dam/GoedgebeurO22}, where this process was used to find witness graphs for new lower bounds on Ramsey numbers.

To find a witness graph for a Ramsey number $R(a_1, a_2)$ on $n$ vertices, one can ask the algorithm above for a graph $G$ on $n$ vertices avoiding $K_{a_1}$ in $G$ and $K_{a_2}$ in the complement of $G$. Such a graph, if it exists, shows that $R(a_1, a_2) > n$ since it is equivalent to a coloring of $K_n$ that has no copy of $K_{a_1}$ in the first color or $K_{a_2}$ in the second color.

In the context of Folkman numbers, witness graphs give upper-bounds: we look for a graph on $n$ vertices that has the arrowing property while avoiding the specified subgraph. 
Thus, we use the generator described above to find semi-polycirculant graphs avoiding the specified subgraph, but the arrowing property is checked using a separate algorithm (details are in the Section~\ref{sec:arrowtest}).
To find witness graph for a Folkman number $F_v(a_1, a_2; H)$ on $n$ vertices, it is clear that we must ask for an $H$-free graph, but we do not have enough information for what subgraph, if any, to avoid in the complement. In theory, we could allow any subgraph in the complement, but this led to an infeasible number of graphs during the generation process. 

Therefore we decided to make an educated guess about the independence number of the desired Folkman graph.
Since one would expect Folkman graphs to be relatively dense, it makes sense to believe that the independence number is not far from the closest Ramsey number.
For example, the \((3,3;J_4;45)^v\)-graph we found has independence number $10$, while \(R(J_4,K_9) \geq 41 \) and \(R(J_4,K_{10})\geq 49\) (best-known lower bounds)~\cite{overbergheThesis}.

We also know that Folkman graphs require a high chromatic number, which is also atypical for graphs with large independent sets.

For the block-structure of the possible Folkman graphs, we usually tried all sequences that were relatively short, because otherwise the number of output graphs would be too high.  For example, $[15,15,5,5]$ would be feasible, but $[12,6,2,2,2,2,2]$ would be too expensive.

\section{Generating Locally Linear Graphs}
\label{sec:locallin}

\subsection{Definitions}

A graph is called \textit{locally linear} (LL) if every edge is contained in exactly one triangle~\cite{fronvcek1989locally}. 
Equivalently, the neighborhood of each vertex is a 1-regular graph. 
Note that $J_4$ contains two triangles sharing an edge, so it is easy to see that any locally linear graph is also $J_4$-free. 
Moreover, when finding graphs that arrow $K_3$, 
it is sufficient to only explore graphs where each edge is part of a triangle. If this is not the case, then this edge is redundant and can be removed. 
Thus, the generation of locally linear graphs is useful when looking for Folkman numbers arrowing $K_3$ and avoiding $J_4$. 
Determining the maximum number of edges (or equivalently: triangles) in an LL-graph with $n$ vertices is an open problem, even asymptotically.

\subsection{Generating All Locally Linear Graphs}

Locally linear graphs have the following recursive structure: if you remove any vertex $v$ and all the edges of all triangles containing $v$, then the resulting graph is also locally linear.
This allows us to generate all locally linear graphs recursively by adding a vertex $v$ and considering all suitable neighborhoods $N(v)$. $N(v)$ should be an independent set of $G$ with $|N(v)|$ even. $N(v)$ must be partitioned into pairs (which form the triangles with $v$) such that the two vertices in each pair have no common neighbors in $G$.
It is a straightforward application of the canonical construction path method to modify this procedure into an algorithm that does not produce isomorphic copies.~\cite{Mc98}
More specifically, we start with a graph consisting of a single vertex. We recursively add a vertex to the graph in every possible way described above, but only once up to equivalence according to $\Aut(G)$.  After the vertex $v$ is added, we calculate whether $v$ was the \textit{canonical} vertex.  We use heuristics, like minimum degree, to decide this very quickly for most graphs.  If all heuristics fail to determine the canonical vertex, we use \textit{nauty} to calculate $\Aut(G)$ and the vertex orbits, which is always sufficient but computationally expensive.

Note that locally linear graphs need not be connected. To compute Folkman numbers, these graphs are obviously not interesting.  More generally, one can assume that a minimal Folkman graph is \textit{maximal locally linear} (MLL): no more triangles can be added such that the graph is still LL. In general, we found the MLL graphs by filtering the LL graphs after they were generated. But we also incorporated heuristics into the generator to specifically target MLL graphs, by avoiding the generation of many non-maximal graphs. For example, after adding the last vertex, we first check for maximality before checking for canonicity since the latter is much more expensive.  This allowed us to compute the MLL graphs one order further than the LL graphs.

We generated all LL graphs up to 22 vertices and MLL graphs up to 23 vertices. The counts are in Table~\ref{table:LL}.

Since $C_4 \subseteq J_4$, minimal Folkman graphs arrowing $K_3$ and avoiding $C_4$ can also be assumed to be locally linear.
An easy modification of the previous algorithm is sufficient to generate only the $C_4$-free graphs.  When adding a new vertex $v$, a 4-cycle appears only when two vertices in $N_v$ are at distance two in the original graph. So we can avoid this by imposing an extra restriction on the candidate neighborhoods. This allowed us to generate all $C_4$-free LL graphs up to 29 vertices, and $C_4$-free MLL graphs up to 30 vertices.

\begin{table}
\scriptsize
	\begin{tabular}{c|cccccc}
		$n$ & $J_4$-free & $\max J_4$-free & LL & MLL & $C_4$-free & $\max C_4$-free \\
		\hline
5 & 22 & 3 & 3 & 1 & 3 & 1 \\
6 & 69 & 6 & 4 & 2 & 4 & 2 \\
7 & 255 & 11 & 6 & 2 & 6 & 2 \\
8 & 1\,301 & 14 & 8 & 3 & 7 & 3 \\
9 & 9\,297 & 56 & 15 & 4 & 12 & 4 \\
10 & 97\,919 & 163 & 22 & 7 & 16 & 7 \\
11 & 1\,519\,456 & 587 & 41 & 9 & 25 & 9 \\
12 & 34\,270\,158 & 2\,922 & 80 & 16 & 35 & 12 \\
13 & 1\,101\,120\,276 & 20\,439 & 201 & 33 & 60 & 20 \\
14 & 36\,508\,415\,316 & 198\,633 & 584 & 83 & 91 & 30 \\
15 &   &   & 2\,498 & 281 & 169 & 53 \\
16 &   &   & 14\,202 & 1\,311 & 288 & 80 \\
17 &   &   & 115\,152 & 8\,527 & 598 & 160 \\
18 &   &   & 1\,235\,365 & 74\,070 & 1\,248 & 284 \\
19 &   &   & 17\,150\,323 & 799\,999 & 3\,153 & 657 \\
20 &   &   & 299\,882\,532 & 10\,661\,060 & 8\,576 & 1\,469 \\
21 &   &   & 6\,539\,648\,206 & 173\,686\,018 & 28\,196 & 4\,427 \\
22 &   &   & 176\,834\,815\,945  & 3\,475\,630\,737 & 104\,388 & 14\,844 \\
23 &   &   &   & 85\,731\,386\,161  & 451\,577 & 59\,343 \\
24 &   &   &   &   & 2\,200\,310 & 259\,395 \\
25 &   &   &   &   & 12\,097\,715 & 1\,296\,145 \\
26 &   &   &   &   & 73\,888\,541 & 7\,219\,888 \\
27 &   &   &   &   & 499\,481\,864 & 44\,423\,361 \\
28 &   &   &   &   & 3\,713\,206\,747 & 299\,707\,178 \\
29 &   &   &   &   & 30\,272\,221\,588  & 2\,218\,262\,426 \\
30 &   &   &   &   &   & 17\,955\,530\,387 \\
	\end{tabular}
	\caption{Counts for $J_4$-free graphs, maximal-$J_4$-free graphs, locally linear graphs (LL), locally linear graphs where no more triangle can be added (MLL), counts of $C_4$-free LL graphs, and counts for $C_4$-free LL graphs where no more triangle can be added without making $C_4$.}
    \label{table:LL}
\end{table}

\section{New Bounds on Folkman Numbers}
\label{sec:results}

In Table~\ref{tab:main}, we show several new and existing bounds on Folkman numbers arrowing $K_2$ and $K_3$ while avoiding $J_k$ and $K_k$ for $k \in \{4,5,6\}$.
Note that the Folkman numbers are listed in non-decreasing order from top to bottom. It is easy to see that it must be so for all the numbers listed except $F_v(2,2,3;H) \leq F_v(3,3;H)$, for which we note the following observation:

\begin{observation}
    $G \ra (3,3)^v \implies G \ra (2,2,3)^v$.
\end{observation}
\begin{proof}
    Suppose for contradiction that $G \not\ra (2,2,3)^v$ and let $c$ be a coloring of $G$ witnessing this. Let $X_i \subseteq V(G)$ be the vertices colored with color $i$ in $c$.
    Observe that $X_1$ and $X_2$ are independent sets.
    Consider the coloring $c'$ of $G$ where the vertices in $X_1 \cup X_2$, and $X_3$ are colored using the first and second color, respectively. 
    Since $X_1$ and $X_2$ are independent sets, $X_1 \cup X_2$ is bipartite and thereby $K_3$-free, so $c'$ must be a monochromatic $K_3$-free coloring of $G$.
    However, this contradicts the fact that $G \ra (3,3)^v$.
\end{proof}

\begin{corollary}
    $F_v(2,2,3;H) \leq F_v(3,3;H)$ for all $H$.
\end{corollary}

In Table~\ref{tab:my-table2} we list our new results for $C_4$-free and $W_5$-free Folkman numbers, as well as two new results for $K_5$-free and $K_6$-free 4-color Folkman numbers.
In Tables~\ref{tab:main} and~\ref{tab:my-table2}, we use markings to show which method was used to find the corresponding bound. For numbers with two markings, the number is discussed in the sections corresponding to both markings.

Our results were obtained by generating/constructing graphs using five methods, then testing the arrowing properties of said graphs. In Section~\ref{sec:arrowtest} we describe how we tested the arrowing properties of graphs.
In Sections \ref{sec:geng}--\ref{sec:special} we discuss the five different approaches we used to obtain our results. In Section~\ref{sec:geng} we discuss results obtained by generating graphs using \nauty's \geng filters. In Section~\ref{sec:ext} we discuss results obtained by extending generated graphs by exploiting special properties. In Sections~\ref{sec:polygen} and~\ref{sec:llgen} we discuss how the algorithms discussed in Sections~\ref{sec:polycirc} and ~\ref{sec:locallin} were used to obtain our results. In Section~\ref{sec:special} we discuss results obtained by using special graphs.
We provide House of Graph (HoG) ID's for our witness graphs.
For each graph for which the HoG ID is provided many of its properties can be seen by accessing the House of Graphs and searching for said ID. Alternatively, if one is reading this paper digitally, the HoG ID's are linked to the corresponding webpage on the House of Graphs.

\subsection{Testing the Arrowing Property}
\label{sec:arrowtest}

For vertex arrowing $(a_1, \ldots, a_k)^v$ we checked whether the vertices of a graph can be colored with $k$ colors, where the $i^{th}$ color induces a $K_{a_i}$-free graph.
This is done with a backtracking algorithm that assigns colors to the vertices one by one.
We first compute a suitable order \((v_1,\dots,v_n)\) of the vertices by breadth-first search.  This ensures that the graphs $G[\{v_1,\dots,v_i\}]$ are all connected (if $G$ is), and hopefully have relatively many edges.  In that way, forbidden cliques can be detected early, and the total branching factor of the algorithm can be reduced.

For the special case of arrowing $(3,3)^v$, we also use a form of forward-checking: after coloring a vertex $v$ blue, we enumerate all triangles $\{v,w,z\}$ where $w$ is also blue and $z$ is not colored. Then we color $z$ red, and repeat the same procedure on $z$.

The only edge arrowing case we focus on is $(3,3)^e$ for which we used SAT solvers \textit{kissat}, \textit{glucose} and \textit{gurobi}~\cite{BiereFallerFazekasFleuryFroleyksPollitt-SAT-Competition-2024-solvers,een2003extensible,gurobi}. By making a variable $x_{e}$ for each edge $e$ of $G$, we construct a CNF formula where each satisfying assignment corresponds to a good coloring of $G$.
To check whether a graph $G$ arrows $(3,3)^e$ we constructed and checked the satisfiablity of the following CNF formula $G_\phi$:
\[ G_\phi = \bigwedge_{(e_1, e_2, e_3)} (x_{e_1} \lor x_{e_2} \lor x_{e_3}) \land (\overline{x_{e_1}} \lor \overline{x_{e_2}} \lor \overline{x_{e_3}}) \]
where each triple $(e_1, e_2, e_3)$ represents the edges of a triangle in $G$. It is clear that a satisfying assignment gives a good coloring: edges assigned ``true'' can be the first color and edges assigned ``false'' can be the second color.

\subsection{Filters \texorpdfstring{$(\filtersym)$}{}}
\label{sec:geng}

The results marked with $\filtersym$ in Tables~\ref{tab:main} and~\ref{tab:my-table2} are computable using \nauty's \geng with (pre-)filters: we wrote programs to check whether a given graph has $H$ as a subgraph. Since \geng generates graphs by adding one vertex at a time, it is possible to stop the construction as soon as the graph contains $H$ as a subgraph.
Note that in each of our cases, we can assume that the extremal graphs are 2-connected. \geng has options to generate only those graphs.
For all the generated graphs, the arrowing property was checked as described in Section~\ref{sec:arrowtest}. Details of the witness graphs found for each case are provided below.

\subsubsection{ \texorpdfstring{ $F_v(2,2,3;J_5) = 10$ }{ Fv(2,2,3;J5) = 10 }}
There are exactly two extremal graphs: the complement of the Petersen graph, and the circulant graph on 10 vertices with distances 2, 3 and 5. Their HoG ID's are \hogid{45703} and \hogid{21154}.

\subsubsection{ \texorpdfstring{ $F_v(3,3;J_5) = 11$ }{ Fv(3,3;J5) = 11 }}
There is exactly one extremal graph.  It has 11 vertices, and can be seen as 3 copies of $K_4$ glued along a vertex, plus two vertices making $K_4$'s with the non-glued vertices. See Figure~\ref{fig:fv33j5}. The HoG ID is \hogid{51287}.

\begin{figure}[t]
    \centering
    \begin{tikzpicture}
        \draw (1.06066017177982, 1.06066017177982) coordinate (x0);
        \draw (1.44888873943360, 0.388228567653781) coordinate (x1);
        \draw (-1.44888873943360, 0.388228567653781) coordinate (x2);
        \draw (-1.06066017177982, 1.06066017177982) coordinate (x3);
        \draw (0.388228567653781, -1.44888873943360) coordinate (x4);
        \draw (-0.388228567653781, -1.44888873943360) coordinate (x5);
        \draw (0.000000000000000, 1.50000000000000) coordinate (x6);
        \draw (1.29903810567666, -0.750000000000000) coordinate (x7);
        \draw (-1.29903810567666, -0.750000000000000) coordinate (x8);
        \draw (0.400000000000000, 0.000000000000000) coordinate (x9);
        \draw (-0.400000000000000, 0.000000000000000) coordinate (x10);

        \draw[edge_color0] (x0)--(x1)--(x6)--(x0)--(x7)--(x1)--(x9)--(x0)--(x10)--(x1);
        \draw[edge_color0] (x2)--(x3)--(x6)--(x2)--(x8)--(x3)--(x9)--(x2)--(x10)--(x3);
        \draw[edge_color0] (x4)--(x5)--(x7)--(x4)--(x8)--(x5)--(x9)--(x4)--(x10)--(x5);
        \draw[edge_color0] (x9)--(x10);
        \draw[edge_color0] (x6)--(x7)--(x8)--(x6);

        \draw (x0) node[unlabeled_vertex]{}; \draw (x1) node[unlabeled_vertex]{}; \draw (x2) node[unlabeled_vertex]{}; \draw (x3) node[unlabeled_vertex]{}; \draw (x4) node[unlabeled_vertex]{}; \draw (x5) node[unlabeled_vertex]{}; \draw (x6) node[unlabeled_vertex]{}; \draw (x7) node[unlabeled_vertex]{}; \draw (x8) node[unlabeled_vertex]{}; \draw (x9) node[unlabeled_vertex]{}; \draw (x10) node[unlabeled_vertex]{};

    \end{tikzpicture}
    \caption{Unique extremal $(3,3;J_5)^v$-graph}
    \label{fig:fv33j5}
\end{figure}
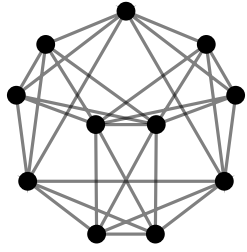

\subsubsection{ \texorpdfstring{ $F_v(2,3,3;J_6) = 12$ }{ Fv(2,3,3;J6) = 12 }}
There are seven extremal graphs. Since $F_v(2,3,3;J_6) = F_v(2,3,3;K_5)$, one of them is the unique extremal graph for the latter Folkman number, whose HoG ID is \hogid{51277}.

\subsubsection{ \texorpdfstring{ $F_v(2,3;C_4) = 17$ }{ Fv(2,3;C4) = 17 }}

There is one extremal graph, which is therefore bicritical. It has two edges that are not part of a triangle. 
Its HoG ID is \hogid{51178}.

\subsubsection{ \texorpdfstring{ $F_e(3,3;J_6) = 11$ }{ Fe(3,3;J6) = 11 }}
There are $803,872,890$ 2-connected $J_6$-free graphs on 11 vertices.  Of these, 3 graphs arrow $(3,3)^e$, which consecutively differ in only one edge from each other.
All of them have a vertex that is connected to all other vertices, but removing it does not give $(3,3)^v$-arrowing graph. Therefore, these give a counterexample to the reverse direction of Observation~\ref{obs:vertexToEdge}.
The HoG-ID is \hogid{51288}.

\subsection{Extending Graphs \texorpdfstring{$(\extsym)$}{}}
\label{sec:ext}

By inspecting the properties of specific Folkman graphs, we discovered necessary conditions for the subgraphs of witness graphs. This allowed us to ``extend'' these subgraphs to obtain witness graphs, if any, for the Folkman number. A graph $G$ is called a $k$-extension of a graph $G'$ if there exists $S \subset V(G)$ of size $k$ such that $G-S$ is isomorphic to $G'$.

\subsubsection{ \texorpdfstring{ $F_v(2,3,3;J_5) \geq 15$ }{ Fv(2,3,3;J5) >= 15 }}

To show that $F_v(2,3,3;J_5) \geq 15$, we make use of the following observation.

\begin{observation}
    Let $G$ be a graph and $S \subset V(G)$ be an independent set. If $G \ra (2,3,3)^v$, then $G-S \ra (3,3)^v$.
\end{observation}
\begin{proof}
    Suppose for contradiction that $G - S \not\ra (3,3)^v$ and let $c$ be a coloring of $G-S$ witnessing this. 
    Let $X_i \subseteq V(G-S)$ be the vertices colored with color $i$ in $c$.
    Consider the coloring $c'$ of $G$ where the vertices in $S$, $X_1$, and $X_2$ are colored using the first, second, and third, color, respectively. 
    Since $S$ is $K_2$-free, and $X_1$ and $X_2$ are $K_3$-free, $c'$ must be a coloring of $G$ that shows $G \not\ra (2,3,3)^v$.
    However, this contradicts the fact that $G \ra (2,3,3)^v$.
\end{proof}

Since $R(J_5, K_3) = 11$ we know that any $J_5$-free graph with at least 11 vertices has an independent set $S$ of size at least three. 
By the observation above, if there exists $G \in \mathcal{F}_v(3,3;J_5;14)$, then we must have $G-S \in \mathcal{F}_v(2,3,3;J_6;11)$.
Thus, 
it is sufficient to test the arrowing property of all $J_5$-free $3$-extensions of the graph in $\mathcal{F}_v(3,3;J_5;11)$. No arrowing graphs were found and thus $F_v(2,3,3;J_5) \geq 15$.
We prove an upper bound $F_v(2,3,3;J_5) \leq 18$ in Section~\ref{sec:polygen}.

\subsubsection{ \texorpdfstring{ $F_v(3,3,3;J_6) = 15$ }{ Fv(3,3,3;J6) = 15 }}

We first note the following observation.

\begin{observation}
    Let $G$ be a graph and $S \subset V(G)$ be an independent set. If $G \ra (3,3,3)^v$, then $G-S \ra (2,3,3)^v$.
\end{observation}
\begin{proof}
    Suppose for contradiction that $G - S \not\ra (2,3,3)^v$ and let $c$ be a coloring of $G-S$ witnessing this. Let $X_i \subseteq V(G-S)$ be the vertices colored with color $i$ in $c$.
    Observe that $X_1$ is an independent set.
    Consider the coloring $c'$ of $G$ where the vertices in $X_1 \cup S$, $X_2$, and $X_3$ are colored using the first, second, and third, color, respectively. 
    Since $X_1$ and $S$ are independent sets, $X_1 \cup S$ is bipartite and thereby $K_3$-free. Thus, $c'$ must be a monochromatic $K_3$-free coloring of $G$.
    However, this contradicts the fact that $G \ra (3,3,3)^v$.
\end{proof}

Note that any graph in 
$\mathcal{F}_v(3,3,3;J_6;14)$
must have an independent set, $S$, of size two, otherwise it is a complete graph and contains a $J_6$. By the observation above, if there exists $G \in \mathcal{F}_v(3,3,3;J_6;14)$, then we must have $G-S \in \mathcal{F}_v(2,3,3;J_6;12)$. 
Thus, 
it is sufficient to test the arrowing property of all $J_6$-free $2$-extensions of the seven graphs in $\mathcal{F}_v(2,3,3;J_6;12)$. No arrowing graphs were found and thus $F_v(3,3,3;J_6) \geq 15$.

To find graphs in $\mathcal{F}_v(3,3,3;J_6;15)$, we took a similar approach by considering two cases. If $G \in \mathcal{F}_v(3,3,3;J_6;15)$ has an independent set of size 3, then it must be a 3-extension of the graphs in $\mathcal{F}_v(2,3,3;J_6;12)$. No arrowing graphs were found in this case and thus, $G \in \mathcal{F}_v(3,3,3;J_6;15)$ cannot have an independent set of size 3.
Thus, any such $G$ must be a $\text{Ramsey}(6,3,12)$-graph; a $\text{Ramsey}(s,t,n)$-graph is a graph with $n$ vertices, no clique of size $s$, and no independent set of size $t$. These have already been enumerated by McKay and we obtained them from McKay's online repository\footnote{\url{https://users.cecs.anu.edu.au/~bdm/data/ramsey.html}}, filtered to obtain $J_6$-free graphs and tested for the arrowing property. One graph was found (see Figure~\ref{fig:333j6}), giving us the upper bound $F_v(3,3,3;J_6)\leq 15$.

Note that the entry in Table~\ref{tab:main} is also marked with $\circsym$. This is because the witness graph was also found using our semi-polycirculant graph generator. The details for this are in the next section.

\subsection{Semi-Polycirculant Graph Generator \texorpdfstring{$(\circsym)$}{}}
\label{sec:polygen}

The upper bounds marked with $\circsym$ were found using the semi-polycirculant graph generator described in Section~\ref{sec:polycirc}. 
The number of vertices, block structure, and HoG ID of each witness graph are provided in Table~\ref{tab:semipoly}. We present more details on these entries below.

\begin{table}[t]
\centering
\begin{tabular}{@{}cccc@{}}
\toprule
Folkman Number       & Upper Bound & Block Structure & HoG ID        \\ \midrule
$F_v (2, 2, 3; J_4)$ & 36          & $[8,8,8,8,4]$   & \hogid{51170} \\
$F_v(3,3;J_4)$       & 45          & $[15,15,15]$    & \hogid{51236} \\
$F_v(2,3,3;J_5)$     & 18          & $[6,6,6]$       & \hogid{51278} \\
$F_v(3,3,3;K_4)$     & 51          & $[17,17,17]$    & \hogid{54024} \\
$F_v(3,3,3;J_5)$     & 32          & $[12,12,6,2]$   & \hogid{53117} \\
$F_v(3,3,3;K_5)$     & 21          & $[9,9,3]$       & \hogid{51286} \\
$F_v(3,3,3;J_6)$     & 15          & $[5,5,5]$       & \hogid{51285} \\
$F_v(2,3,3,3;K_5)$   & 32          & $[16,16]$       & \hogid{53087} \\
$F_v(3,3,3,3;K_6)$   & 30          & $[15,15]$       & \hogid{51330} \\ \bottomrule
\end{tabular}
\caption{New upper bounds for Folkman numbers using our semi-polycirculant graph generator.
For each upper bound, we provide the HoG ID of a graph witnessing said bound.
Note that two of the bounds listed here were also found using other methods: (1) $F_v(3,3,3;J_6)$ was also found by the extension method described in Section~\ref{sec:ext}, and (2) $F_v(3,3,3;J_5)$ was also found by modifying special graphs as described in Section~\ref{sec:special}.}
\label{tab:semipoly}
\end{table}

\begin{itemize}
    \item The witness graph $G$ for $F_v(3,3;J_4)$ is locally linear and vertex-transitive. All vertex-transitive graphs up to order 47 were checked and the aforementioned graph was the only one with the desired property.
Moreover, all polycirculant graphs on at most 4 equal blocks with $\alpha < 11$ were checked, and no smaller graphs were found. We note that $G$ has a deeper structure, that can be seen in different ways:
\begin{itemize}
    \item $G$ has an automorphism for which all vertex orbits are triangles. If we remove the edges of those triangles, we obtain the ``halved Foster graph''~\cite{MR1761910}.
    \item Geometrically: $G$ is associated with the partial linear space called $\widetilde{W_2}$, a triple cover of the generalized quadrangle $GQ(2,2)=W_2$.
    Specifically, the vertices of $G$ are the points of the geometry. For every line (which have 3 points), and for every \textit{fiber} (i.e., the equivalent points in the triple cover) we add a triangle in $G$.
    This geometry has several natural constructions, which are described in~\cite{tildeGeometry}.
\end{itemize}
\item The witness graph for $F_v(2,3,3;J_5)$ is a ``triple cover'' of $K_{2,2,2}$.
\item 
For $F_v(3,3,3;K_4)$,
we checked all vertex-transitive graphs with orders between 32 and 47, but there were no Folkman graphs among them with these parameters.
We also verified that none of the $\approx 2\cdot 10^{10}$ $K_4$-free 3-polycirculant graphs with $\alpha < 9$ on 48 vertices arrow $(3,3,3)^v$.

We note that the previous bound on this number was 66~\cite{deng2013upper}, which in turn implied bounds on the numbers $F_v(6,6,6;7) \leq 726$ and $F_e(3,3,3;8)$ $ \leq 727$ using the following Lemma:
\begin{lemma}[\cite{kolev2008multiplicative,deng2013upper}]
    $F_v(6,6,6;7) \leq F_v(2,2,2;3) \times F_v(3,3,3;4)$, and 
    $F_e(3,3,3;8) \leq F_v(6,6,6;7) + 1$.
\end{lemma}
Our result that $F_v(3,3,3;K_4) \leq 51$
implies the following:
\begin{corollary}
    $F_v(6,6,6;7) \leq 561$ and $F_e(3,3,3;8) \leq 562$.
\end{corollary}
\item For $F_v(3,3,3;J_5)$, we note that 
polycirculant graphs with fewer blocks that we tried do not get close to the upper bound we found. 
For example, it took 8 CPU days to generate and check all $\approx 6.6\cdot 10^9$ $J_5$-free 3-polycirculant graphs on 42 vertices with $\alpha \leq 7$, and we verified that none were witness graphs.
\item For $F_v(3,3,3;K_5)$, we note that 
all of the vertex-deleted subgraphs of our witness graph do not arrow $(3,3,3)^v$.

\item Our witness graph for $F_v(3,3,3;J_6)$ can be seen as a five-vertex extension of the complement of the Petersen graph: 
Let $X:=\{1,2,3,4,5\}$, $V_1:=\binom{X}{1}$ and $V_2:=\binom{X}{2}$. \(G=(V_1\cup V_2,E)\), with 
\[
\{v_1,v_2\}\in E \Leftrightarrow
\begin{cases}
    v_1\cap v_2 \neq \emptyset & \text{if } v_1,v_2 \in V_2 \\
    v_1 \cap v_2 = \emptyset   & \text{otherwise}
\end{cases}
\]
It is depicted in Figure~\ref{fig:333j6}.
This graph also arrows $(J_4,J_5)^v$.

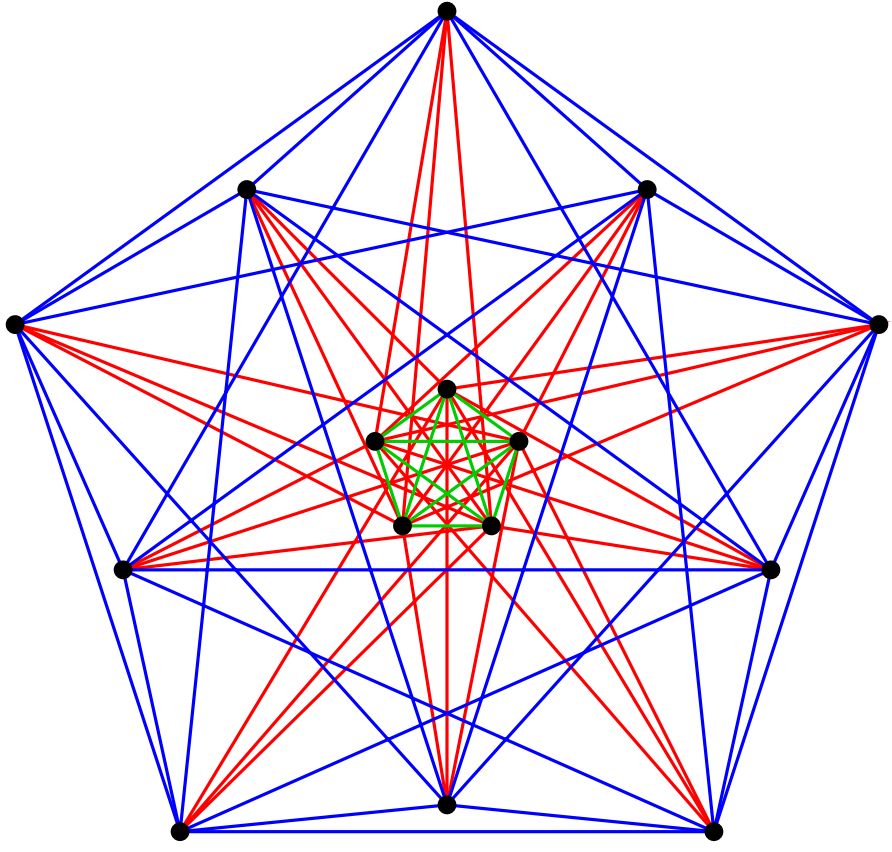
\begin{figure}[t]
    \centering
\begin{tikzpicture}[scale=0.5]
	\draw (-1.17557050458495, -1.61803398874989) coordinate (x0);
	\draw (-1.90211303259031, 0.618033988749895) coordinate (x4);
	\draw (0.000000000000000, 2.00000000000000) coordinate (x3);
	\draw (1.17557050458495, -1.61803398874989) coordinate (x1);
	\draw (1.90211303259031, 0.618033988749895) coordinate (x2);
	\draw (-7.05342302750968, -9.70820393249937) coordinate (x5);
	\draw (-11.4126781955418, 3.70820393249937) coordinate (x9);
	\draw (5.29006727063226, 7.28115294937453) coordinate (x13);
	\draw (11.4126781955418, 3.70820393249937) coordinate (x7);
	\draw (8.55950864665638, -2.78115294937453) coordinate (x12);
	\draw (-8.55950864665638, -2.78115294937453) coordinate (x10);
	\draw (0.000000000000000, -9.00000000000000) coordinate (x11);
	\draw (0.000000000000000, 12.0000000000000) coordinate (x8);
	\draw (7.05342302750968, -9.70820393249937) coordinate (x6);
	\draw (-5.29006727063226, 7.28115294937453) coordinate (x14);

	\draw[edge_color1] (x5)--(x1)--(x8)--(x0)--(x7)--(x3)--(x5)--(x2)--(x6)--(x3)--(x11)--(x0)--(x9)--(x1)--(x10)--(x2)--(x9);
	\draw[edge_color1] (x6)--(x4)--(x7);
	\draw[edge_color1] (x8)--(x4)--(x10);
	\draw[edge_color1] (x11)--(x2)--(x13)--(x0)--(x14)--(x1)--(x12)--(x3)--(x14);
	\draw[edge_color1] (x12)--(x4)--(x13);
	
	\draw[edge_color2] (x10)--(x12)--(x14)--(x11)--(x13)--(x10)--(x5)--(x11)--(x6)--(x10)--(x8)--(x12)--(x5)--(x14)--(x7)--(x11)--(x9)--(x10);
	\draw[edge_color2] (x12)--(x6)--(x13)--(x7)--(x12);
	\draw[edge_color2] (x13)--(x8)--(x14)--(x9)--(x13);
	\draw[edge_color2] (x5)--(x6)--(x7)--(x8)--(x9)--(x5);
	
	\draw[edge_color3] (x0)--(x1)--(x2)--(x0)--(x3)--(x1)--(x4)--(x0);
	\draw[edge_color3] (x2)--(x3)--(x4)--(x2);
	
	\draw (x0) node[unlabeled_vertex]{}; \draw (x1) node[unlabeled_vertex]{}; \draw (x2) node[unlabeled_vertex]{}; \draw (x3) node[unlabeled_vertex]{}; \draw (x4) node[unlabeled_vertex]{}; \draw (x5) node[unlabeled_vertex]{}; \draw (x6) node[unlabeled_vertex]{}; \draw (x7) node[unlabeled_vertex]{}; \draw (x8) node[unlabeled_vertex]{}; \draw (x9) node[unlabeled_vertex]{}; \draw (x10) node[unlabeled_vertex]{}; \draw (x11) node[unlabeled_vertex]{}; \draw (x12) node[unlabeled_vertex]{}; \draw (x13) node[unlabeled_vertex]{}; \draw (x14) node[unlabeled_vertex]{};
	
\end{tikzpicture}
\caption{The unique graph in $\mathcal{F}_v(3,3,3;J_6;15)$ discussed in Sections~\ref{sec:ext} and~\ref{sec:polygen}. The colors are for clarity only: in blue is the complement of the Petersen graph, in green is a $K_5$, in red are the connections between the two subgraphs.}
\label{fig:333j6}

\end{figure}

\item For $F_v(2,3,3,3;K_5)$, we note that there are no 4-polycirculant graphs on 28 vertices with $\alpha < 5$, and no 3-polycirculant on 30 vertices with $\alpha < 6$ with the arrowing property.

\item The witness graph for $F_v(3,3,3,3;K_6)$ has chromatic number equal to $10$, which is is higher than strictly required.
The smallest vertex-transitive graph with this property has 31 vertices.
This result cannot be improved with 2- or 3-polycirculant graphs with $\alpha < 6$.
\end{itemize}

\subsection{Locally Linear Graph Generator \texorpdfstring{$(\llsym)$}{}}
\label{sec:llgen}

Using the locally linear graph generator described in Section~\ref{sec:locallin}, we were able to generate all locally linear graphs on up to 22 vertices, and all maximally locally linear graphs on 23 vertices. Recall that a locally linear graph is $J_4$-free and each edge belongs to a $K_3$. Moreover, if $G \ra (3,3)^v$ and $e \in E(G)$ does not belong to a $K_3$, then $G - e \ra (3,3)^v$. Thus, if there exists a graph in $\mathcal{F}_v(3,3;J_4;n)$, then there exists a maximally locally linear graph in 
$\mathcal{F}_v(3,3;J_4;n)$.
We checked the arrowing property for each maximally locally linear graph on $23$ vertices and found no arrowing graph, giving us that $F_v(3,3;J_4) \geq 24$.

For a lower bound on $F_v(3,3;C_4)$, we modified our generator to produce $C_4$-free graphs. This allowed to generate all $C_4$-free locally linear graphs on up to 28 vertices, and all $C_4$-free maximally locally linear graphs on 29 vertices. We checked the arrowing property for each such graph on 29 vertices and found no arrowing graph, giving us that $F_v(3,3;C_4) \geq 30$.

\subsection{Modifying Special Graphs $(\modsym)$}
\label{sec:special}

In this section we discuss new upper bounds found by modifying some special graphs that are already known in the literature. 

\subsubsection{ \texorpdfstring{ $F_v(3,3,3;J_5) \leq 32$ }{ Fv(3,3,3;J5) <= 32 }}
\label{mod:v333j5}

We started from the graph 
\(TO^{-}(6,2) \)~\cite{brouwer2022strongly}.
 \(TO^{-}(6,2) \) is a vertex- and edge-transitive graph of order 56.
The neighborhood of each vertex is isomorphic to the complement of the Schl\"afli graph (\(=O^{-}(6,2)\)). It is the only connected graph with this property, together with $VO^{-}(6,2)$, which is the unique extremal Ramsey $(J_5,J_7)$-graph~\cite{DBLP:journals/dam/GoedgebeurO22}.
\(TO^{-}(6,2)\) is also isomorphic to the complement of the \textit{Gosset graph} with one perfect matching removed.

This graph arrows $(3,3,3)^v$ and is $J_5$-free.
For each vertex $v \in V(G)$, we obtained the graph $G - v$, removing any isomorphs, and checking if the arrowing property still holds. This property was iteratively repeated on all smaller graphs until no more arrowing graphs were found.

This process could be repeated 24 times, requiring to check the arrowing property of around one million graphs, and resulting in exactly one graph of order 32.
This graph is semi-polycirculant with block structure [12,12,6,2].
It does not arrow $(3,3)^e$. Its HoG ID is \hogid{53117}.

\subsubsection{ \texorpdfstring{ $F_e(3,3;J_5) \leq 43$ }{ Fe(3,3;J5) <= 43 }} \label{ss:33J5}

We started from the graph from the previous section: \(TO^{-}(6,2) \).
By transforming \((K_3,K_3)\)-arrowing to \textit{3-SAT} and using the \textit{glucose} SAT-solver, we were able to show that \(TO^{-}(6,2) \rightarrow (3,3)^e \).

It was observed that up to 13 vertices from \(TO^{-}(6,2)\) can be removed so that the arrowing property still holds. This was done by first removing a maximum independent set (7 vertices), and then greedily removing 6 more vertices. The resulting graph $G$ has $43$ vertices, 440 edges, 1093 triangles, \(\chi(G)=7\), \(\alpha(G)=7\), \(|Aut(G)|=48\). All of its vertex-deleted subgraphs are non-arrowing.
The HoG ID of our witness graph is \hogid{51171}.
We note that the previous upper bound for $F_e(3,3;J_5)$ was 136 \cite{HassanJNRX23}, and now we have 43.

Other attempts to find smaller examples have so far failed. For example, all vertex-transitive graphs up to $48$ vertices were enumerated by Royle and Holt~\cite{DBLP:journals/jsc/HoltR20}. We verified using the same methods as above that no \(J_5\)-free vertex-transitive graph on 36, 39, 40, 42 or 44 vertices arrows \((3,3)^e\).

\subsubsection{ \texorpdfstring{ $F_v(3,3;C_4) \leq 63$ }{ Fv(3,3;C4) <= 63 }}
\label{sec:fv33c4}

We explored a family of graph based on projective planes~\cite{MR2014512}.
Take a projective plane of order $q$, where the points $P$ are labeled by triples over $\mathbb{F}_q$.  Define the \textit{polarity graph} $G_q$ as the graph $(V, E)$ where $V=\{v\in P\ \vert\ v\cdot v \neq 0\}$ and $E=\{\{v,w\} \vert \  v\cdot w = 0 \}$, using the standard inner product. $G_q$ is a locally linear $C_4$-free graph.
This graph can also be constructed as \(V=\mathbb{F}_8^2\setminus (0,0) \) and 

\[(x_1,y_1)\sim (x_2,y_2) \Leftrightarrow  \begin{vmatrix}
	x_1 & y_1 \\ x_2 & y_2
\end{vmatrix} = 1. \]

Then it is easy to see that the only common neighbor of \((x_1,y_1)\) and \((x_2,y_2) \) is \((x_1+x_2,y_1+y_2)\).
 
It was observed that $G_8\rightarrow (3,3)^v$, with $|V(G_8)|=63$, $\chi(G_8) = 5$ and $\alpha(G_8)=15$. $G_8$ is vertex-transitive. Every vertex-deleted subgraph does not arrow. 
The graph is not MLL. Its HoG ID is \hogid{51177}.

\subsubsection{ \texorpdfstring{ $F_e(3,3;W_5) \leq 64$ }{ Fe(3,3;W5) <= 64 }}

In the following, we formally explain an interesting link between vertex and edge Folkman numbers that avoid monochromatic $K_3$'s. This link
allowed us to obtain some new bounds for Folkman numbers.
While this observation has been noted in the literature before, we provide a formal proof for it below.

\begin{observation}
    Suppose that $G$ is a graph and $G'$ is the graph constructed by adding a vertex $v$ to $G$ and adding an edge $(u,v)$ to every $u \in V(G)$.
    If $G \ra (3,3)^v$, then $G' \ra (3,3)^e$.
    \label{obs:vertexToEdge}
\end{observation}
\begin{proof}
    We prove via contradiction. Suppose $G' \not\ra (3,3)^e$ and consider an edge-coloring, $c$, of $G'$ with no red and no blue $K_3$'s. 
    Let $R \subseteq V(G')$ (resp., $B \subseteq V(G'))$ be the set of vertices in $V(G')$ that $v$ shares a red (resp., blue) edge with in $c$.

    Observe that any edge between two vertices $u,w \in R$ must be blue; otherwise $(u,v)$, $(w,v)$, and $(u,w)$ form a red $K_3$ in $c$. Since $c$ is a good coloring, it also follows that $R$ is $K_3$-free. A similar argument shows that $B$ is also $K_3$-free. Consider the vertex-coloring $c'$ of $G$ where each vertex in $R$ is colored red and each vertex in $B$ is colored blue. Since $R$ and $B$ are $K_3$-free, there is no red and no blue $K_3$ in either subset of vertices. Any $K_3$ that uses vertices from both sets is not monochromatic. 
    Thus, $c'$ is a vertex coloring of $G$ with no red and no blue $K_3$'s.
    However, this contradicts the fact that $G \ra (3,3)^v$.
\end{proof}

\begin{observation}
    If $G$ is a $C_4$-free graph, then $G$ $+$ $v$ is $W_5$-free.
\end{observation}
\begin{proof}
    If $G$ $+$ $v$ contains a copy of $W_5$, it would have to include $v$. But the removal of any vertex in $W_5$ always leaves a $C_4$, which would then be completely inside $G$.
\end{proof}

Therefore, if we add a vertex connected to all other vertices of a graph in $\mathcal{F}_v(3,3;C_4;n)$, we obtain a graph in $\mathcal{F}_e(3,3;W_5;n+1)$.
Using the bound from Section~\ref{sec:fv33c4} gives us $F_e(3,3;W_5) \leq F_v(3,3;C_4) + 1 \leq 64$.

We note that the existence of a $W_5$-free graph that arrows $(3,3)^e$ was an open problem~\cite{HassanJNRX23}.
Now the only 5-vertex graph $H$ for which the existence of $F_e(3,3;H)$ is unknown is $\ptt$.
This case cannot be solved with the same methodology, since it would require \((3,3)\)-vertex arrowing of a graph without $K_3+e$, which is impossible.

\section{Proof of Theorem~\ref{thm:ptt}}
\label{sec:ptt}

In this section we restate our theorem about $F_e(3,3;\ptt)$ and provide proofs of our claims.

\pttThm*

\begin{proof}
    Let $G$ be a $\ptt$-free graph that arrows $(3,3)^e$ such that for any edge $e \in E(G)$, $G -e \not\ra (3,3)^e$, i.e., $G$ is minimal.

\begin{enumerate}
        \item 
                Suppose $G$ has a $K_4$, and let $X$ be the vertices of said $K_4$. Let $v \in V(G) - X$. If $v$ is adjacent to two vertices in $X$, then a $\ptt$ is formed. Thus, for all $v \in V(G) - X$, $v$ is adjacent to at most one vertex in $X$. 
        In the context of $(3,3)^e$ arrowing, this means that the edges in $X$ can be colored independently of the other edges. Therefore,
        we could remove them from $G$ without changing whether it arrows $(3,3)^e$, contradicting $G$'s minimality.
        
        \item 
                Let $e \in E(G)$.
        By minimality, $G-e \not\ra(3,3)^e$. 
        If $e$ is in at most one triangle, then
        it is easy to see that any good coloring of $G-e$ can be extended to a good coloring of $G$.

        \item 
                If $G[N(u)]$ contains a $K_3$, $G$ contains a $K_4$, contradicting property 1.
        If $G[N(u)]$ contains a $C_4$, $G$ contains $\ptt$.
        \item \textit{Suppose $u$ is a vertex in $G$. Then, for each edge $e$ in $G[N(u)]$, there exists a vertex $v \in X = V(G) - N(u) - \{u\}$ that both of $e$’s endpoints are adjacent to.}
        This follows from (2) and (3).
        If $v$ was connected to both edges in $G[N(u)]$ then this would form a $\ptt$.
        
        \item 
Consider a minimal $\ptt$-free graph $G$ that arrows $(3,3)^e$, and an arbitrary vertex $v$ in $G$. If every 2-coloring of $G[N(v)]$ without monochromatic triangles can be extended to a 2-coloring of $G[N(v) \cup \{v\}]$ without creating a monochromatic triangle, then $v$ is a redundant vertex.

In order to find a lower bound on the minimal degree of a minimal $\ptt$-free, $(3,3)^e$-arrowing graph $G$, we executed the following procedure with a computer:
\begin{itemize}
    \item Consider each graph $H$ on $n$ vertices that does not contain $K_3$, and for which adding a universal vertex would not create a copy of $\ptt$.
    \item For each such graph, consider every possible 2-coloring of its edges.
    \item For each such coloring, check if it can be extended with one extra vertex so that there are no monochromatic triangles.
\end{itemize}

The smallest graph for which a coloring exists that cannot be extended has 8 vertices. It is a graph obtained from $K_4$, where the edges of one 4-cycle have all been subdivided by a vertex. \qedhere
    \end{enumerate}
\end{proof}

\section{Conclusion}
\label{sec:conclude}

We presented many new bounds on Folkman numbers arrowing $K_2$ and $K_3$ while avoiding special graphs on up to six vertices. Notably, a majority of our new upper bounds are from generating witness graphs using our semi-polycirculant graph generator, showcasing its efficacy for finding witness Folkman graphs. 

Our work showed the existence of $F_e(3,3;W_5)$, leaving only one case open for the existence of $F_e(3,3;H)$ where $H$ is a graph on five vertices, $H = \ptt$. For this case we provided some properties of the minimal Folkman witness graph. We close our paper with the following open question:

\begin{problem}
    Does $F_e(3,3;\ptt)$ exist?
\end{problem}

\bibliographystyle{alpha}
\bibliography{references}

\end{document}